\documentclass{amsart}

\usepackage{amsmath,amssymb, amsthm}
\usepackage{marvosym}
\usepackage[english]{babel}
\usepackage{graphicx}
\usepackage{caption}
\usepackage{etoolbox}
\patchcmd{\quote}{\rightmargin}{\leftmargin 2em \rightmargin}{}{}

\usepackage{url}
\theoremstyle{plain}
\newtheorem{theo}{Theorem}[section]
\newtheorem{lem}[theo]{Lemma}
\newtheorem{claim}[theo]{Claim}

\newtheorem{cor}[theo]{Corollary}

\newcommand{\R}{\mathbb R}
\theoremstyle{definition}
\newtheorem{definition}{Definition}
\providecommand{\comment}[1]{}

\newcommand{\sub}{\mathrm{Sub}}

\newcommand{\PSL}{\mathrm{PSL}}

\newcommand{\RR}{\mathbb R}

\author[I. Biringer]{Ian Biringer}
\author[J. Raimbault]{Jean Raimbault}
  \address{Institut de math\'ematiques de Toulouse, Universit\'e Paul Sabatier\\ Toulouse, France}
  \email{Jean.Raimbault@math.univ-toulouse.fr}

\title{Ends of unimodular random manifolds}

\begin{document}

\maketitle

\section{Introduction}

A \emph{unimodular random $d$-manifold}, or URM, is a random element $(M,p)$ of the space $\mathcal M^d$ of all pointed  (complete, connected) Riemannian $d$-manifolds, whose law $\nu$ is a `unimodular' probability measure on $\mathcal M^d$, see Definition \ref{unimodulardef}.  Informally, the unimodularity means that for a each fixed manifold $M $, the basepoints $p \in M$ are distributed  with respect to the Riemannian volume of $M$.  As explained in detail in \cite{Abert_Biringer}, see also \S \ref{definitionssec}  below, URMs simultaneously generalize finite volume manifolds, their regular covers, randomly chosen leaves in measured foliations, and invariant random subgroups of the isometry groups of Riemannian manifolds. 
In this short paper, we  describe the topology of the space of ends of a URM.  

\vspace{2mm}

An end of a Riemannian manifold $M$  is called \emph {finite volume}  if it has a finite volume neighborhood;  otherwise, it is \emph{infinite volume.}  The infinite volume ends of $M$  form a closed subset $\mathcal E_\infty(M) \subset \mathcal E(M)$ of the space of all ends. See \S \ref{ends}. 

\begin {theo}\label{main}
 If $(M,p)$ is a URM, then the following hold almost surely:
\begin{enumerate}
\item either $|\mathcal E_\infty(M)|=0,1,2$ or $\mathcal E_\infty(M)$  is a  Cantor set,
	\item if $M$ has a finite volume end, every infinite volume end of $M$ is a limit of finite volume ends,
\item setting the dimension $d=2$, if $M$ is a surface with genus $g(M)>0$, every infinite volume end of $M$ has infinite genus.
\end{enumerate}\end {theo}

Note that the topology of a finite volume end  of a Riemannian manifold can be completely arbitrary.  However, if $S$ is a surface with bounded curvature,  it is well-known that every finite volume end of $S$ is isolated and has a neighborhood homeomorphic to an open annulus (Lemma \ref{finvolends}). Using the classification  of infinite type surfaces, see \S \ref{ends}, Theorem~\ref {main}  implies the following:

\begin{cor}
Suppose $(S,p)$ is an  orientable unimodular random surface with bounded curvature.  Then almost surely, either 
\begin {enumerate}
	\item $S$ has finite volume and finite  topological type,
\item $S$ has infinite volume and is homeomorphic to one of  the following 12 surfaces: the cylinder, the plane,  one of the four surfaces in Figure \ref{surfacefig}, or a surface obtained by puncturing one of these six surfaces at a locally finite set of points that intersects all end neighborhoods.
\end {enumerate}
 Furthermore, all these topological types can be realized by such a URS.
\label{IRS_top_main}
\end{cor}

\begin {figure}[h]
\captionsetup{width=.9\linewidth}
\centering
\includegraphics{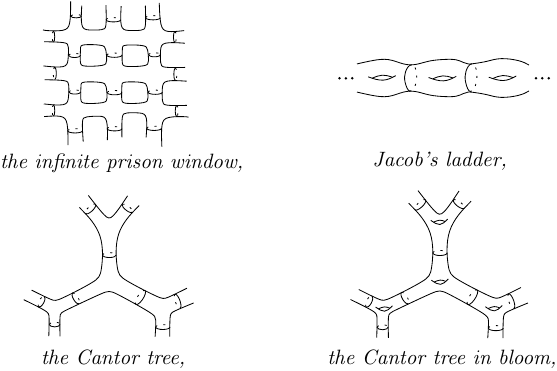}
\caption{The Cantor tree is the boundary of a regular neighborhood of a 3-valent tree properly embedded in $\mathbb R^3$.  In the spring, genus blooms near each of its vertices.  The infinite prison window is also homeomorphic to the \emph {Loch Ness monster} of  \cite{Ghys_feuille} and the \emph {one-ended Jacob's ladder}.}\label {surfacefig}
\end {figure} 

There is a natural extension of Corollary \ref{IRS_top_main} for non-orientable surfaces: the conclusion is that $S$  is  almost surely either finite volume, one of the four surfaces of Figure \ref{surfacefig} with sequences of non-orientable loops (and possibly punctures) exiting all its ends, or a cylinder or plane with \emph{both}  punctures and non-orientable loops exiting its ends. This  can be proved with the same techniques as Corollary \ref{IRS_top_main}, but we will only consider orientable surfaces here for simplicity.

\subsection{Motivation, definitions and applications}  
\label {definitionssec} As mentioned above, let 
$\mathcal M^d$ be the space of all complete, connected, pointed Riemannian $d$-manifolds $(M,p)$,  taken up to pointed isometry. We regard $\mathcal M^d$ in the \emph{smooth topology}, where two pointed manifolds  are close if their metrics are $C^{\infty}$-close on large diffeomorphic neighborhoods of their base points. See \cite {Abert_Biringer} for details. 
Similarly, let $\mathcal M^d_2$ be the space of  all \emph {doubly pointed} Riemannian  manifolds $(M,p,q) $, taken up to doubly pointed isometry and considered in the appropriate smooth topology. 

\begin {definition}\label {unimodulardef}
A Borel probability measure $\nu$ on $\mathcal M^d $ is \emph {unimodular} if and only if for every  nonnegative Borel function $f : \mathcal M^d_2 \longrightarrow \RR$ we have 
\begin{equation}
\int_{(M,p) \in \mathcal M^d} \int_{q\in M}  f(M,p,q) \, dvol \, d\nu= \int_{(M,p) \in \mathcal M^d} \int_{q \in M} f(M,q,p) \, dvol \, d\nu. 
\label{MTP}
\end{equation}
A \emph {unimodular random $d$-manifold} (or \emph {surface}, when $d=2$) is a random pointed  manifold $(M,p)$  whose law is a unimodular probability measure on $\mathcal M^d $.\end {definition}
Equation \ref {MTP} is usually called the \emph {mass transport principle}, or MTP.  It has its roots in foliations, but the setting in which we use it is motivated by graph theory, see Aldous-Lyons \cite{Aldousprocesses} and \cite{Benjaminirecurrence,Benjaminigroup,Haggstrominfinite}.  As an example  of a unimodular measure, suppose $M$ is a finite volume Riemannian surface, and let $\mu_M$ be the measure on $\mathcal M^d$ obtained by pushing forward the  Riemannian measure $ vol$ under  the map
\[M \longrightarrow \mathcal M^d, \ \ \ p \longmapsto (M,p).\]
 Then both sides of the MTP are  equal to the integral of $f(M,p,q)$ over $(p,q)\in M \times M$, so the measure $\mu_M $ is unimodular, \cite{Abert_Biringer}.

 For a more general example, let $X $ be a \emph{foliated space}, a  separable metrizable space $X$ that is a union of `leaves' that fit together locally as the horizontal factors in a product $\R^d \times Z$ for some transversal space $Z$.  Suppose $X$ is \emph{Riemannian}, i.e.\ the leaves all have Riemannian metrics, and these metrics vary smoothly in the transverse direction, see \cite[\S 3]{Abert_Biringer}. There is then a \emph{leaf map}  $$X \longrightarrow\mathcal M^d, \ \ x \longmapsto (L_x,x),$$ where $L_x$  is the leaf of $X$ through $x$.  
Let $\mu$ be a \emph{completely invariant}  probability measure on $X $: a measure obtained by integrating the Riemannian measures on the leaves of $X $ against some invariant transverse measure, see \cite{Candelfoliations2}.   The push forward of $\mu$ under the leaf map  is a unimodular measure on $\mathcal M ^ d$. See \cite{Abert_Biringer} for details.

 Combining this  discussion with Theorem \ref{main}, we get:

\begin {theo}[Generic leaves] Whenever $X$ is a Riemannian foliated space with a completely invariant probability measure $\mu$,  the leaf through $\mu$-almost every  point $x\in X$  satisfies the conclusions (1) -- (3) of Theorem \ref{main}.	\label {genericleaves}
\end {theo}

 There is a large amount of literature  available concerning generic leaves of Riemannian foliated spaces, see for instance  the papers by Cantwell--Conlon \cite{Cantwellendsets,CantwellendsetsII,Cantwellgeneric}, Ghys~\cite{Ghys_feuille} and \'Alvarez L\'opez--Candel \cite{Alvarezgeneric}, and also \cite{Alvarezturbulent,Bermudezlaminations,Hectorends,Walczakdynamics} and \cite[Chapter 3]{Candelfoliations2}.    At least when $X$ is compact, parts (1) and (3) of Theorem~\ref{genericleaves} follow from Ghys's 1995 work \cite{Ghys_feuille} on foliations  endowed with \emph{harmonic measures},  which generalize completely invariant measures. Ghys also uses classification of surfaces to 
prove the harmonic foliated version\footnote{Again, Ghys assumes $X $  is compact. This rules out cusps in the leaves, so that the only possible infinite type surfaces in Ghys's theorem are the cylinder, the plane and those of Figure \ref{surfacefig}.} of  Corollary \ref{IRS_top_main}.

We doubt that Theorem \ref{genericleaves} is very surprising  to those working in the field, but what is interesting to us is that its proof (even  together with the  translation from foliations to URMs in \cite{Abert_Biringer})  is extremely simple.  In some sense, the MTP captures exactly the  recurrence property necessary to prove these \emph {`if it happens somewhere, it happens everywhere'} statements efficiently.

\vspace{2mm}

 Another way to construct unimodular measures on $\mathcal M^d$  is through quotients by \emph {invariant random subgroups}. If $G$ is a  locally compact group, let $\sub_G$ be the space of closed subgroups of $G$ endowed with the Chabauty topology, see \cite[Section 2]{abert2012growth}.  
\begin {definition}
An \emph {invariant random subgroup} (IRS) of $G$ is a random closed subgroup whose law is a  conjugation invariant probability measure on $\sub_G$.  
\end {definition} 
A simple example is the IRS whose law is the unique $G$-invariant probability measure $\nu_\Gamma$ on the conjugacy class of a lattice $\Gamma<G$; more involved examples are given in \cite[Section 12]{abert2012growth}.  
IRSs were first studied by Ab\'ert-Glasner-Vir\'ag \cite {Abertkesten} for discrete $G$, see also Vershik \cite{Vershiktotally},  and were introduced to Lie groups in \cite{abert2012growth}.

In \cite[\S 2.2]{Abert_Biringer}, it is shown that when $X$ is a complete Riemannian  manifold and $G < \mathrm{Isom}(X)$  is a subgroup such that either
\begin {enumerate}
\item $G$ is unimodular and act transitively on $X $, or
\item $G$ acts freely and properly discontinuously, and $G \backslash X$ has finite volume,
\end {enumerate}
 then any IRS $\Gamma<G$ that acts freely and properly discontinuously gives a unimodular random manifold $M=\Gamma\backslash X $,  where the base point is the projection of a fixed point in $X$ in case (1), and is  the projection of a randomly chosen point from  a fundamental domain for the action $G \circlearrowright X$ in case (2).

 All the  conclusions of Theorem \ref{main} then hold for  these IRS quotients. So in particular, setting $G = \PSL_2 \R$ and $X = \mathbb H^2$,  we get the following, which was our initial motivation for writing this paper and was explained in the survey \cite{grenoble}.

\begin {cor}[Topology of IRS quotients]
 Suppose that  $\Gamma $ is  a discrete, torsion free IRS of $\PSL_2 \R$.  Then almost surely, the quotient $S= \Gamma \backslash \mathbb H^2$  is  either  finite type or is homeomorphic to one of the 12 surfaces that appear in  Theorem \ref{IRS_top_main}.
\end {cor}

 Actually, only 11 surfaces appear: IRSs of $\PSL_2 \R$  are concentrated on subgroups with full limit set, by \cite[Proposition 11.3]{abert2012growth}, so cannot be cyclic.

Note that any normal subgroup $H$  of a group $G$  can be considered as an IRS. So, by (2) above, any regular cover $\hat S$ of a (say, closed) orientable surface $S$ can be considered as a unimodular random surface, after fixing a Riemannian metric on $S$ and  randomly choosing a base point for $\hat S$  as above. In this special case, Theorem~\ref{main}~(1) is the classical theorem of Hopf \cite{Hopfenden} on ends of groups,  and the  corresponding special case of Corollary \ref{IRS_top_main}  was noticed by Grigorchuk \cite {Grigorchuktopological}.

\subsection{Acknowledgements}   We would like to thank Miklos Abert for helpful conversations. The first author was partially supported by NSF grant DMS 1611851.

\section{Ends, bounded curvature and  the proof of Corollary \ref{IRS_top_main}}
\label{ends}
 Suppose that $X$ is a topological space admitting a countable cover by compact sets, e.g. any locally compact, separable space. The \emph {space of ends} ${\mathcal E}(X)$ is then defined as the inverse limit of the system of complements of compact subsets of $X$. It is a compact, totally disconnected topological space. 

More concretely, one can construct ${\mathcal E}(X)$ as follows.  Choose a nested sequence of compact subsets $K_1\subset K_2\subset\ldots$ in $X$ such that $X = \bigcup_i K_i$. A point in ${\mathcal E}(X)$ is determined by a sequence $(C_i)$, where each $C_i$ is a component of $X\setminus K_i$ and $C_{i+1}\subset C_i$. Moreover, for every $i$ and complementary component $C_i$ we get a map ${\mathcal E}(C_i)\to{\mathcal E}(X)$, and the images of these maps are a basis for the topology. 

If $X$ is an orientable surface, one can take the $K_i$ to be subsurfaces and define the genus of an end $(C_i)$ as the limit of the genus of the $C_i$. This is either zero or infinity, and the set of ends with infinite genus is a closed set.  In fact, a surface is completely determined by the genus and topology of its ends; this was originally proven by B. Ker\'ekj\'art\'o, but a modern proof is given by I. Richards in \cite{Richards_surfaces}:

\begin {theo}[Classification of noncompact surfaces]
Suppose that $S$ and $T$ are orientable surfaces with the same genus and that there is a homeomorphism $$\phi : {\mathcal E}(S) \to {\mathcal E}(T)$$ such that for all $\xi\in{\mathcal E}(S), $ the genera of $\xi$ and $\phi(\xi)$ are the same.  Then $S$ and $ T$ are homeomorphic.
\end {theo}

 As  mentioned in the introduction,  the classification of surfaces  combines with Theorem \ref{main} to give a  classification of topological types of unimodular random surfaces of bounded curvature.  The key lemma is the following:

\begin {lem}[Finite volume ends]\label{finvolends} Suppose that $S$ is a  complete Riemannian surface  with bounded Gaussian curvatures.  Then every finite volume end of $S$ has a neighborhood homeomorphic to an open annulus, and so is isolated in $\mathcal E(X)$.
\end {lem}
\begin {proof}
 This is a well-known  corollary of	 the `Good Choppings' theorem of Cheeger--Gromov \cite{Cheegerchopping}.  We may assume  that $S $ itself has finite volume, by arbitrarily replacing the complement of a  finite volume end neighborhood with something compact.

By \cite[Theorem 0.5]{Cheegerchopping}, $S$ is a nested union of compact subsurfaces $S_1,S_2,\ldots$ that are disjoint except along their boundaries, such that the boundary curves $\partial S_i$ have length  tending to zero and  uniformly bounded geodesic curvatures $\kappa$.
As $S$ has finite volume, we have $\mathrm{vol}(S_{i+1} \setminus S_i) \to 0$, and since curvature is bounded,
$$\chi(S_{i+1} \setminus S_i) = \int_{S_{i+1} \setminus S_i} K \, dvol  \, - \, \int_{\partial S_i} \kappa \, ds \to 0.$$
So for large $i$, we have that $S_{i+1} \setminus S_i$ is a union of annuli, and the lemma follows.
\end {proof}

To prove Corollary \ref{IRS_top_main}, then, note that if $S$ is any bounded curvature surface satisfying the conclusion of Theorem \ref{main}, then either $S$  has finite volume,
\begin{enumerate}
\item $S$ has genus zero and  $\mathcal E(S)$  consists of either $1$ or $2$ genus zero ends,
\item $S$ has infinite genus and $\mathcal E(S)$  consists of either $1$ or $2$  infinite genus ends,
\item $\mathcal E(S)$ is a Cantor set, all of genus $0$ ends,
 \item $\mathcal E(S)$ is a Cantor set, all of genus $\infty$ ends,
\end{enumerate}
or $\mathcal E(S)$ is  as described in one of the four cases above, except with  isolated genus $0$ ends accumulating on to all the previously described ends.  By classification of surfaces, this means that $S $ is as described in Theorem \ref{IRS_top_main}.

To realize all the  listed topological types,  we recall from \cite{Abert_Biringer}, see also \S \ref{definitionssec}, that any  regular cover of a finite volume surface can be regarded as a URS.   The listed surfaces are all  clearly homeomorphic to such covers.

\section{The proof of Theorem \ref{main}}
Throughout this section, let $\nu$ be a unimodular measure on $\mathcal M^d $.   For clarity, we will work directly with $\nu$ here  rather than referencing a $\nu$-random element of $\mathcal M^d$.
We will prove the three assertions of Theorem \ref{main} in turn.  Each is a quick application of the mass transport principle (see Definition \ref{unimodulardef}).

\vspace{2mm}

First, we want to show that for $\nu$-a.e.\ $(M,p)$,  we have that $|\mathcal E_\infty(M)|=0,1,2$ or $\mathcal E_\infty(M)$  is a  Cantor set.  As $\mathcal E_\infty(M)$ is compact and totally disconnected, it suffices to show that it is perfect.  In other words, we claim:


\begin{lem}\label {isolatedlemma}
The following holds for $\nu$-a.e.\ $(M,p) \in \mathcal M^d$: if $M$ has an  infinite volume end  that is isolated in $\mathcal E_\infty(M)$, then $|\mathcal E_\infty(M)|\leq 2$.
\label{isolated}
\end{lem}

\begin{proof}
 Throughout the proof, we can assume that $\nu$ is concentrated on infinite volume manifolds.  
If a  pointed manifold $(M,p) $ has  at least three infinite volume ends, one of which is isolated in $\mathcal E_\infty(M)$, then there are integers $r,V>0$  such that
\begin {enumerate}
\item $M \setminus B_M(p,r)$ has at least three infinite volume connected components,
\item some component  of $ M \setminus B_M(p,r)$  has a single infinite volume end,
\item  the volume of $B_M(p,2r) $  is at most $ V$.
\end {enumerate}
 So,  it suffices to show that for each $r,V$, it is $\nu$-almost  never the case that (1) -- (3) hold.  Define a function $f : \mathcal M^d_2 \longrightarrow \{0,1\}$ by setting $f(M,p,q)=1$ if  conditions (1) -- (3)  are satisfied for $(M,p)$ and if $q$ lies in  one of the components from (2).

\begin{claim}\label {quick}
This function $f$ is Borel.
\label{f_is_Borel}
\end{claim}
\begin{proof}
For each $A>0$, the condition that $ M \setminus B_M(p,r)$ has at least three components with volume bigger than $A$, one of which contains $q$, is an open condition on $(M,p,q)\in \mathcal M^d $.  So  requiring (1) and (2), and also that $q$ lies in an  infinite volume component of $ M \setminus B_M(p,r)$,  is an intersection of open conditions, and hence is Borel. Condition (3), that  $B_M(p,2r) \leq V$, is closed.
  
 Similarly,  for any $R>r$, the condition that the component of $ M \setminus B_M(p,r)$ containing $q$ contains at least two components of $M\setminus B_M(p,R )$ with infinite volume is Borel.  Therefore, the condition $f(M,p,q)=1$ is the complement of a countable union of Borel conditions, and therefore is Borel.
\end {proof}

\begin{claim}If $(M,q) \in \mathcal M^d$, then $\mathrm{vol}\{ p\in M \ | \ f(M,p,q)=1 \} \leq V$.\label {volumeestimate}
\end{claim}

\begin{proof}
It  suffices to show that $B(p_1,r) \cap B(p_2,r) \neq \emptyset$ whenever $f(M,p_i,q)=1$; since then  after fixing $p_1$,  we  always have $p_2 \in B(p_1,2r)$, so  the volume estimate follows from condition (3). 
So, suppose  the balls $B(p_1,r) $ and $B(p_2,r)$ are disjoint, and that $q$  is contained in components $C_i$  of their complements.  Then either
\begin{enumerate}
	\item $B(p_2,r)$  separates $q$ from $p_1$,  in which case  there  can only be two  infinite volume components of $M \setminus B(p_2,r)$, as in Figure \ref{components} (a), or
\item $B(p_2,r)$  does not separate $q$ from $p_1$,    in which case either $C_1$ or $C_2$  has more than one infinite volume end, as in Figure \ref{components} (b) and (c).
\end{enumerate} 
 In both cases, though, this is a contradiction.
\end {proof}
\begin {figure}[t]
\centering\includegraphics {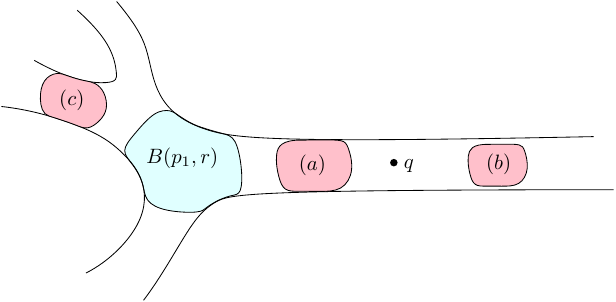}
\caption {Here, (a), (b) and (c)  indicate possible positions for $B(p_2,r)$, relative to $B(p_1,r)$.}
\label {components}
\end {figure}

The mass transport principle (MTP), see Definition \ref {unimodulardef}, now states:
$$\int_{(M,q) \in \mathcal M^d} \int_{p\in M}  f(M,p,q) \, dvol \, d\nu= \int_{(M,q) \in \mathcal M^d} \int_{p \in M} f(M,q,p) \, dvol \, d\nu. $$
By Claim \ref{volumeestimate}, the inner integral on the left side of the MTP is always at most $V$, so the left side of the MTP  is at most $V$. On the other hand, if $(M,q) \in \mathcal M^d $  satisfies (1) -- (3), the inner integral on the right side of the MTP is clearly infinite.  So,  conditions (1) -- (3) are $\nu$-almost never satisfied, which proves Lemma \ref {isolatedlemma}.\end{proof}


 We now verify the second part of  Theorem \ref {main}.

\begin{lem}
The following holds for $\nu$-a.e.\ $(M,p)\in \mathcal M^d$: if $M$ has a  finite volume  end then every infinite volume end  of $M $ is a limit of finite volume ends.
\label{cusp_limit}
\end{lem}

\begin{proof}
 The proof is similar to that of Lemma \ref{isolated}, and we assume that the reader has understood its proof in what follows. Fixing $r,v,V>0$, the function $f : \mathcal M^d_2 \longrightarrow \{0,1\}$  we input in the MTP  is defined by setting $f(M,p,q)=1$  when \begin {enumerate}
\item $q$ belongs to an infinite-volume connected component of $M\setminus B_M(p,r)$ that has no  finite volume ends,
\item  some unbounded component of $M\setminus B_M(p,r)$  has volume at most $v$,
\item $vol \, B_M(p,r) > v$, and $vol \, B_M(p,2r) \leq V$. \end {enumerate} 
Using arguments similar to those in Claim \ref{f_is_Borel},  it is easy to show that $f$ is Borel.  It suffices, then, to show that for fixed $(M,q) $, the volume of the set of all $p\in M$ for which (1) -- (3) are satisfied is at most $V$,  for then the lemma follows from the MTP in the same way that Lemma \ref{isolated} did.

 This also works the same way that it did in Claim \ref{volumeestimate}.  Bearing condition (3)  in mind, we only have to show that if $f(M,p_1,q)=f(M,p_2,q)=1$ then $B_M(p_1,r) \cap B_M(p_2,r) \neq \emptyset$.  Assume they are  disjoint, and let $C_1,C_2$   be the  complementary components containing $q$.  We now break into two cases.
\begin {enumerate}
	\item $B_M(p_2,r) \subset C_1$.   In this case, any unbounded, finite volume component $K \subset M\setminus B_M(p_2,r)$ with  must contain $B_M(p_1,r)$, since otherwise it is  contained in $C_2$,  which by (1) has no finite volume ends.  Choosing $K$  so that $vol \,  K \leq v$, this contradicts that $vol \, B_M(p_1,r) > v$.
\item $B_M(p_2,r) \not \subset C_1$.   Suppose that $K \subset M \setminus B_M(p_1,r)$  is an unbounded component with volume at most $v$.  By volume considerations, $B_M(p_2,r) $  is not contained in $K$. Hence, $K$ and $C_1$  are  both contained in $ C_2$,  which is a contradiction since $C_2$ has no finite volume ends.\qedhere
\end {enumerate}
\end{proof}

%

 Finally, we set the dimension $d=2$,  and prove the third part of  Theorem \ref{main}.

\begin{lem}
The following holds  for $\nu$-a.e.\ $(M,p)\in \mathcal M^2$: if $M$ has  an infinite volume end of genus zero, then $M$ has genus zero. 
\label{genus}
\end{lem}

\begin{proof}
The proof is again similar to that of Lemma \ref{isolated}  and one should understand the  proofs above before reading  further.  Fixing $r,V>0$, set $f(M,p,q)=1$ if   
\begin{enumerate}
\item $q$  is contained in a component of $M\setminus B_M(p,r)$  that is infinite volume  and has genus zero,
\item $B_M (p, r) $ contains a subsurface with positive genus
\item $vol \, B_M (p, 2r) \leq V$. 
\end{enumerate}
The proof that $f$ is Borel is similar to that of Claim \ref{f_is_Borel}. It is easy to see that whenever $f(M,p_1,q)=f(M,p_2,q)=1$ then $B_M(p_1,r) \cap B_M(p_2,r) \neq \emptyset$, for if $C_1, C_2$ are the   complementary components containing $q$, then conditions (1) and (2) imply that neither $B_M(p_1,r) \subset C_2$ nor $B_M(p_2,r) \subset C_1$. As before, this implies that the  set of all $p$  such that $f(M,p,q)=1$ has volume at most $V$, and the proof concludes  using the MTP in the same way as it does in Lemma \ref{isolated}.  
\end{proof}
\comment {The condition that $q$ lies in 

We pick a triple $(M,p,q)$ such that $f(M,p,q)=1$ and a neighbourhood $W$ of it such that for all $(M',p',q')\in W$ we have that $(M',p')$ satisfies $(\ast)_r$ and $d(p',q')>r$. We let $C'$ be the connected component of $M\setminus B_{M'}(p,r)$ containing $q$ and for $R>r$ we define
$$
U_R = \{ (M',p',q')\in W :\: C'\cap B_M(p',R) \text{ is of genus 0} \}.
$$
Then $U_R$ is open, and thus
$$
f^{-1}(1)\cap W = \bigcap_{R>r} U_R
$$
is a Borel set. }


\bibliographystyle{plain}
\bibliography{bib,bibrefs}

\begin{thebibliography}{10}

\bibitem{abert2012growth}
Mikl{\'o}s Ab{\'e}rt, Nicolas Bergeron, Ian Biringer, Tsachik Gelander, Nikolay
  Nikolov, Jean Raimbault, and Iddo Samet.
\newblock On the growth of ${L}^2$-invariants for sequences of lattices in
  {L}ie groups.
\newblock {\em arXiv:1210.2961}, 2012.

\bibitem{Abert_Biringer}
Mikl{\'o}s Ab{\'e}rt and Ian Biringer.
\newblock Unimodular measures on the space of all {R}iemannian manifolds.
\newblock {\em https://arxiv.org/abs/1606.03360}, 2016.

\bibitem{Abertkesten}
Mikl{\'o}s Ab{\'e}rt, Yair Glasner, and B{\'a}lint Vir{\'a}g.
\newblock Kesten's theorem for invariant random subgroups.
\newblock {\em arXiv:1201.3399}, 2012.

\bibitem{Aldousprocesses}
David Aldous and Russell Lyons.
\newblock Processes on unimodular random networks.
\newblock {\em Electron. J. Probab.}, 12:no. 54, 1454--1508, 2007.

\bibitem{Alvarezturbulent}
Jes\'us~A. \'Alvarez~L\'opez and Alberto Candel.
\newblock Turbulent relations.
\newblock 2012.

\bibitem{Alvarezgeneric}
Jes\'us~A. \'Alvarez~L\'opez and Alberto Candel.
\newblock Generic coarse geometry of leaves.
\newblock 2014.

\bibitem{Benjaminigroup}
I.~Benjamini, R.~Lyons, Y.~Peres, and O.~Schramm.
\newblock Group-invariant percolation on graphs.
\newblock {\em Geom. Funct. Anal.}, 9(1):29--66, 1999.

\bibitem{Benjaminirecurrence}
Itai Benjamini and Oded Schramm.
\newblock Recurrence of distributional limits of finite planar graphs.
\newblock {\em Electron. J. Probab.}, 6:no. 23, 13 pp. (electronic), 2001.

\bibitem{Bermudezlaminations}
Miguel Berm{\'u}dez and Gilbert Hector.
\newblock Laminations hyperfinies et rev\^etements.
\newblock {\em Ergodic Theory Dynam. Systems}, 26(2):305--339, 2006.

\bibitem{Candelfoliations2}
Alberto Candel and Lawrence Conlon.
\newblock {\em Foliations. {II}}, volume~60 of {\em Graduate Studies in
  Mathematics}.
\newblock American Mathematical Society, Providence, RI, 2003.

\bibitem{Cantwellendsets}
John Cantwell and Lawrence Conlon.
\newblock Endsets of leaves.
\newblock {\em Topology}, 21(4):333--352, 1982.

\bibitem{Cantwellgeneric}
John Cantwell and Lawrence Conlon.
\newblock Generic leaves.
\newblock {\em Comment. Math. Helv.}, 73(2):306--336, 1998.

\bibitem{CantwellendsetsII}
John Cantwell and Lawrence Conlon.
\newblock Endsets of exceptional leaves; a theorem of {G}. {D}uminy.
\newblock In {\em Foliations: geometry and dynamics ({W}arsaw, 2000)}, pages
  225--261. World Sci. Publ., River Edge, NJ, 2002.

\bibitem{Cheegerchopping}
Jeff Cheeger and Mikhael Gromov.
\newblock Chopping {R}iemannian manifolds.
\newblock In {\em Differential geometry}, volume~52 of {\em Pitman Monogr.
  Surveys Pure Appl. Math.}, pages 85--94. Longman Sci. Tech., Harlow, 1991.

\bibitem{Ghys_feuille}
{\'E}tienne Ghys.
\newblock Topologie des feuilles g\'en\'eriques.
\newblock {\em Ann. of Math. (2)}, 141(2):387--422, 1995.

\bibitem{Grigorchuktopological}
R.~I. Grigorchuk.
\newblock Topological and metric types of surfaces that regularly cover a
  closed surface.
\newblock {\em Izv. Akad. Nauk SSSR Ser. Mat.}, 53(3):498--536, 671, 1989.

\bibitem{Haggstrominfinite}
Olle H{\"a}ggstr{\"o}m.
\newblock Infinite clusters in dependent automorphism invariant percolation on
  trees.
\newblock {\em Ann. Probab.}, 25(3):1423--1436, 1997.

\bibitem{Hectorends}
Gilbert Hector, Shigenori Matsumoto, and Ga{\"e}l Meigniez.
\newblock Ends of leaves of {L}ie foliations.
\newblock {\em J. Math. Soc. Japan}, 57(3):753--779, 2005.

\bibitem{Hopfenden}
Heinz Hopf.
\newblock Enden offener {R}\"aume und unendliche diskontinuierliche {G}ruppen.
\newblock {\em Comment. Math. Helv.}, 16:81--100, 1944.

\bibitem{grenoble}
Jean Raimbault.
\newblock G{\'e}om{\'e}trie et topologie des vari{\'e}t{\'e}s hyperboliques de
  grand volume.
\newblock {\em S{\'e}minaire de Th{\'e}orie spectrale et g{\'e}om{\'e}trie
  (Grenoble)}, 31, 2012-2014.

\bibitem{Richards_surfaces}
Ian Richards.
\newblock On the classification of noncompact surfaces.
\newblock {\em Trans. Amer. Math. Soc.}, 106:259--269, 1963.

\bibitem{Vershiktotally}
Anatolii~Moiseevich Vershik.
\newblock Totally nonfree actions and the infinite symmetric group.
\newblock {\em Moscow Mathematical Journal}, 12(1):193--212, 2012.

\bibitem{Walczakdynamics}
Pawe{\l} Walczak.
\newblock {\em Dynamics of foliations, groups and pseudogroups}, volume~64 of
  {\em Instytut Matematyczny Polskiej Akademii Nauk. Monografie Matematyczne
  (New Series) [Mathematics Institute of the Polish Academy of Sciences.
  Mathematical Monographs (New Series)]}.
\newblock Birkh\"auser Verlag, Basel, 2004.

\end{thebibliography}

\end{document}